\documentclass{amsart}
\usepackage[margin=3.5cm]{geometry}
\usepackage{amsmath,amssymb}

\usepackage{tabularx}
\usepackage{mathtools} 

\usepackage{xcolor}
\usepackage{hyperref}
\definecolor{darkgreen}{rgb}{0,0.4,0}
\definecolor{BrickRed}{rgb}{0.65,0.08,0}
\hypersetup{colorlinks=true,linkcolor=blue,citecolor=darkgreen,filecolor=BrickRed,urlcolor=darkgreen}

\newtheorem{theorem}{Theorem}[section]
\newtheorem{proposition}[theorem]{Proposition}

\newtheorem{lemma}[theorem]{Lemma}

\theoremstyle{definition}
\newtheorem{remark}{Remark}
\numberwithin{equation}{section}
\newcommand{\tO}{\mathtt 0}                       
\newcommand{\tL}{\mathtt 1}                       

\DeclareMathOperator{\e}{\mathrm{e}}                
\newcommand{\LandauO}{\mathcal{O}}                  
\newcommand{\Landauo}{\mathrm{o}}                   
\newcommand{\E}{\mathbb E}

\newcommand{\lbracket}{\left[\hspace{-0.4em}\left[}
\newcommand{\rbracket}{\right]\hspace{-0.4em}\right]}

\allowdisplaybreaks     
\title[Divisibility of binomial coefficients]%
{Divisibility of binomial coefficients by powers of two}
\author{Lukas Spiegelhofer}
\address{Institut f\"ur diskrete Mathematik und Geometrie,
Technische Universit\"at Wien,
Wiedner Hauptstrasse 8--10, 1040 Wien, Austria}
\author{Michael Wallner}
\address{Laboratoire d'Informatique de Paris Nord, Universit\'e Paris 13, 99 Avenue Jean-Baptiste Cl\'ement, 93430 Villetaneuse, France}
\thanks{
The first author acknowledges support by the project MUDERA (Multiplicativity, Determinism, and Randomness),
which is a joint project between the FWF (Austrian Science Fund) and the ANR (Agence Nationale de la Recherche), and by project F5502-N26 (FWF), which is a part of the Special Research Program ``Quasi Monte Carlo Methods: Theory and Applications''.
The second author acknowledges support by Project SFB F50-03, which is a part of the Special Research Program ``Algorithmic and Enumerative Combinatorics''.
}


\begin{document}
\maketitle
\begin{abstract}
For nonnegative integers $j$ and $n$ let $\Theta(j,n)$ be the number of entries in the $n$-th row of Pascal's triangle that are not divisible by $2^{j+1}$.
In this paper we prove that the family $j\mapsto\Theta(j,n)$ usually follows a normal distribution.
The method used for proving this theorem involves the computation of first and second moments of $\Theta(j,n)$, and uses asymptotic analysis of multivariate generating functions by complex analytic methods, building on earlier work by Drmota (1994) and Drmota, Kauers and Spiegelhofer (2016).
\end{abstract}
\pagestyle{myheadings}
\thispagestyle{plain}
\markboth{SPIEGELHOFER AND WALLNER}{DIVISIBILITY OF BINOMIAL COEFFICIENTS}
\section{Introduction}
Divisibility of binomial coefficient by powers of primes is a notion strongly linked to the base-$p$ expansion of integers.
This connection is highlighted by Kummer's famous result~\cite{K1852} stating that the highest power $m$ of a prime $p$ dividing a binomial coefficient $\binom nt$ equals the number of \emph{borrows} occurring in the subtraction $n-t$ in base $p$.
No less well-known is Lucas' congruence~\cite{L1878}:
\begin{equation}\label{eqn:lucas}
\binom nt\equiv\binom{n_{\nu-1}}{t_{\nu-1}}\cdots \binom{n_0}{t_0}\bmod p,
\end{equation}
where $n=(n_{\nu-1}\cdots n_0)_p$ and $t=(t_{\nu-1}\cdots t_0)_p$ are the expansions of $n$ and $t$ in base $p$.
A wealth of classical results related to divisibility of binomial coefficients can be found in Dickson's book~\cite{D1919}.
More recent surveys concerning binomial coefficients modulo prime powers were written by Granville~\cite{G1997} and Singmaster~\cite{Si1980}.
We note that Kummer's theorem has been generalized to $q$-multinomial coefficients by Fray~\cite{F1967} and to generalised binomial coefficients by Knuth and Wilf~\cite{KW1989}.
Also, Lucas' congruence has been extended in different directions, see~\cite{DW1990,F1967,G1992,G1997,K1968}.

The main object of study in the present paper is the number of binomial coefficients $\binom nk$ exactly divisible by a power of $2$.
More generally, for nonnegative integers~$j$ and $n$ and a prime~$p$ we define
\begin{equation}\label{eqn_def_row_count}
\vartheta_p(j,n)
\coloneqq\left\lvert\left\{t\in\{0,\ldots,n\}:\nu_p\left(\!\binom nt\!\right)=j\right\}\right\rvert,
\end{equation}
where $\nu_p(m)$ is the largest $k$ such that $p^k\mid m$.
Moreover, we define partial sums:

\begin{equation}\label{eqn_def_row_count2}
\Theta_p(j,n)
\coloneqq\sum_{0\leq i\leq j}\vartheta_p(i,n)
=\left\lvert\left\{t\in\{0,\ldots,n\}:2^{j+1}\nmid\binom nt\right\}\right\rvert.
\end{equation}

Lucas' congruence yields a formula for the case $j=0$ (see Fine~\cite{F1947}):
\begin{equation}\label{eqn:fine}
\vartheta_p(0,n)=
\prod_{0\leq i<\nu}(n_i+1)
=2^{\lvert n\rvert_1}3^{\lvert n\rvert_2}4^{\lvert n\rvert_3}\cdots p^{\lvert n\rvert_{p-1}},
\end{equation}
where $\lvert n\rvert_d$ is the number of times the digit $d$ occurs in the base-$p$ expansion of $n$.
In particular, writing $s_2(n)=\lvert n\rvert_1$ for the binary sum-of-digits function, we obtain (see Glaisher~\cite{G1899})
\begin{equation}\label{eqn_glaisher}
\vartheta_2(0,n)=
2^{s_2(n)}.
\end{equation}
For $j\geq 1$, the quantities $\vartheta_p(j,n)$ and $\Theta_p(j,n)$ can be expressed using block-counting functions.
For a finite word $w$ on the symbols $0,\ldots,p-1$, containing at least one symbol $\neq 0$, and a nonnegative integer $n$, we define $\lvert n\rvert_w$ as the number of times the word $w$ occurs as a contiguous subword of the binary expansion of $n$.
It was proved by Rowland~\cite{R2011}, and implicitely by Barat and Grabner~\cite{BG2001}, that $\vartheta_p(j,n)/\vartheta_p(0,n)$ is given by a polynomial $P_j$ in the variables $X_w$,
where~$w$ are certain finite words in $\{0,\ldots,p-1\}$,
and each variable $X_w$ is set to $\lvert n\rvert_w$.
For example, we have the following formulas, found by Howard~\cite{H1971}:
\begin{align*}
\frac{\vartheta_2(1,n)}{\vartheta_2(0,n)}&=\frac 12\lvert n\rvert_{\tL\tO}\\
\frac{\vartheta_2(2,n)}{\vartheta_2(0,n)}&=
-\frac 18\lvert n\rvert_{\tL\tO}
+\frac 18\lvert n\rvert_{\tL\tO}^2
+\lvert n\rvert_{\tL\tO\tO}
+\frac 14\lvert n\rvert_{\tL\tL\tO},\\[2mm]
\frac{\vartheta_2(3,n)}{\vartheta_2(0,n)}&=
\frac 1{24}\lvert n\rvert_{\tL\tO}
-\frac 1{16}\lvert n\rvert_{\tL\tO}^2
-\frac 12\lvert n\rvert_{\tL\tO\tO}
-\frac 18\lvert n\rvert_{\tL\tL\tO}
+\frac 1{48}\lvert n\rvert_{\tL\tO}^3
+\frac 12\lvert n\rvert_{\tL\tO}\lvert n\rvert_{\tL\tO\tO}\\
&+\frac 18\lvert n\rvert_{\tL\tO}\lvert n\rvert_{\tL\tL\tO}
+2\lvert n\rvert_{\tL\tO\tO\tO}
+\frac 12\lvert n\rvert_{\tL\tO\tL\tO}
+\frac 12\lvert n\rvert_{\tL\tL\tO\tO}
+\frac 18\lvert n\rvert_{\tL\tL\tL\tO}.
\end{align*}
The number of terms in these expressions is sequence \texttt{A275012} in Sloane's OEIS~\cite{OEIS} and can be seen as a measure of complexity of the sequence $n\mapsto \vartheta_2(j,n)$. This was noted by Rowland (see the comments to \texttt{A001316}, \texttt{A163000} and \texttt{A163577} in the OEIS).
In the recent paper~\cite{SW2016} the authors prove a structural result on the polynomials representing $\vartheta_p(j,n)/\vartheta_p(0,n)$,
which also allows to compute them efficiently.
We note that the above representation as a polynomial implies that $n\mapsto \vartheta_p(j,n)$ is a $p$-regular sequence in the sense of Allouche and Shallit~\cite{AS1992}.

The above formulas are \emph{exact}; in this paper we want to consider matters from a more analytical point of view and we are interested in asymptotic properties of divisibility of binomial coefficients.
For some asymptotic results on binomial coefficients modulo primes and prime powers we refer the reader to the papers by Holte~\cite{H1997}, Barat and Grabner~\cite{BG2001,BG2016} and Greinecker~\cite{G2017}.
Our main result is related to a theorem proved by Singmaster~\cite{S1974} saying that any given integer~$d$ divides almost all binomial coefficients.
From his Theorem~1(C) it follows that, for all $j\geq 0$,
\begin{equation}\label{eqn_singmaster}
\frac 1N\sum_{0\leq n<N}\frac{\Theta_p(j,n)}{n+1}=\Landauo(1).
\end{equation}
(Note that there are $n+1$ elements in row number $n$ of Pascal's triangle.)
This behaviour is clearly different from the divisibility pattern of the sequence of positive integers: we have
\[
\lim_{N\rightarrow \infty}
\frac 1N \left\lvert\left\{0\leq n<N:
\nu_2(n)\leq j
\right\}\right\rvert
=
\lim_{N\rightarrow \infty}
\frac 1N \left\lvert\left\{0\leq n<N:
2^{j+1}\nmid n
\right\}\right\rvert
= 1-\frac 1{2^{j+1}}.
\]

We are interested in the ``typical'' divisibility of a binomial coefficient:
our main theorem states that the probability distribution defined by
$j\mapsto \Theta_2(j,n)/(n+1)$ usually is close to a normal distribution
with expected value $\log n/\log 2-s_2(n)$ and variance $\log n/\log 2$.
This is a refinement of the case $p=2$ of Equation~\eqref{eqn_singmaster}. 

A related result, concerning \emph{columns} of Pascal's triangle, was proved by Emme and Hubert~\cite{EH2016}, continuing work by Emme and Prikhod'ko~\cite{EP2015}.
In that paper, Emme and Hubert consider the quantity
\[\mu_a(d)=\lim_{N\rightarrow\infty} \frac 1N
\left\lvert \{n<N:s_2(n+a)-s_2(n)=d\}\right\rvert\]
and they prove a central limit type theorem for these values.
Note that the connection to columns in Pascal's triangle is given by the identity
$s_2(n+a)-s_2(n)=s_2(a)-\nu_2\binom{n+a}a$, which can be derived from Legendre's relation $\nu_2(n!)=n-s_2(n)$.

\textit{Notation.}
In this paper, $s_2(n)$ denotes the binary sum-of-digits function, that is, the number of $\tL$s in the binary expansion of $n$.
Moreover, $\nu_p(m)$ is the maximal $k$ such that $p^k\mid m$.
We write $\nu_p\binom nk\coloneqq\nu_p\bigl(\binom nk\bigr)$.
\section{The main result}
Let $\Phi(x)=\frac 1{\sqrt{2\pi}}\int_{-\infty}^x \e^{-t^2/2}\,\mathrm d t$
and set $\Theta_2(j,n)=0$ for $j<0$.
For convenience, we define $\Theta_p(x,n)\coloneqq \Theta_p(\lfloor x\rfloor,n)$ for real $x$.
Then the following theorem holds.
\begin{theorem}\label{thm_main}
Assume that $\varepsilon>0$.
For an integer $\lambda\geq 0$ we set 
$I_\lambda=[2^\lambda,2^{\lambda+1})$. 
Then
\begin{equation*}
\left \lvert\left\{n\in I_\lambda:
\left \lvert \frac{\Theta_2(\lambda-s_2(n)+u,n)}{n+1}-\Phi\left(\frac u{\sqrt{\lambda}}\right)\right \rvert\geq \varepsilon\text{ for some }u\in\mathbb R\right\}\right \rvert
=\LandauO\left(\frac {2^\lambda}{\sqrt{\lambda}}\right),
\end{equation*}
where the implied constant may depend on $\varepsilon$.

\end{theorem}
Informally, for most $n\in I_\lambda$ the distribution function defined by $k\mapsto \Theta(k,n)/(n+1)$ follows a normal distribution with expected value $\lambda-s_2(n)$ and variance $\lambda$.
\begin{remark}
Let $n$ be a nonnegative integer and define the expected value
\begin{equation}\label{eqn:average_divisibility}
\mu_n=\frac 1{n+1}\sum_{k\geq 0}k\vartheta(k,n).
\end{equation}

We have
\begin{align*}
\mu_n&=
\frac 1{n+1}\sum_{k\geq 0}\sum_{0\leq t\leq n}k
\lbracket \nu_2\binom nt=k\rbracket
=
\frac 1{n+1}\sum_{0\leq t\leq n}\nu_2\binom nt
\\&=
\frac 2{n+1}\sum_{0\leq t\leq n}s_2(t)-s_2(n)
\end{align*}
by the identity
$\nu_2\binom{n}{t}=s_2(n-t)+s_2(t)-s_2(n)$ we noted before.
Here $\lbracket \,S\,\rbracket$ denotes the Iverson bracket which is $1$ if the statement $S$ is true, and $0$ otherwise.
Using Delange~\cite{D1975}, we obtain the representation
\[\mu_n=\log(n+1)/\log 2-s_2(n)+F\bigl(\log (n+1)/\log 2\bigr),\]
where $F$ is a continuous function of period $1$.
If $2^\lambda\leq n<2^{\lambda+1}$, we have therefore
$\mu_n=\lambda-s_2(n)+\LandauO(1)$,
which is consistent with Theorem~\ref{thm_main}.
\end{remark}
\begin{remark}
Due to the recurrence relation underlying the values $\vartheta_2(j,n)$ (see Section~\ref{sec_rec}) the intervals 
$I_\lambda=\left[2^\lambda,2^{\lambda+1}\right)$ 
are the easiest to work with. 
However, we can extend our result to intervals $[0,N)$ 
by concatenating intervals $I_\lambda$:
we obtain
\begin{equation*}
\left \lvert\left\{n<N
:\sup_{u\in\mathbb R}
\left \lvert \frac{\Theta(\lfloor \log_2 n\rfloor-s_2(n)+u,n)}{n+1}-\Phi\left(\frac u{\sqrt{\log_2 n}}\right)\right \rvert\geq \varepsilon\right\}\right \rvert
=\LandauO\left(\frac N{\sqrt{\log N}}\right),
\end{equation*}
where $\log_2 n=\log n/\log 2$.
We skip the details of the proof.
\end{remark}
\textit{Idea of the proof of Theorem~\ref{thm_main}.}
The essential idea is to show that, for given $u$ and $\varepsilon>0$,
 $\Theta(\lambda-s_2(n)+u,n)/(n+1)$ is $\varepsilon$-close to
$\Phi(u/\sqrt{\lambda})$ for all but few      
$n\in[2^\lambda,2^{\lambda+1})$. 
If this is achieved, we can perform this approximation for all $u\in K$, where
$K$ is evenly spaced in $[-R\sqrt{\lambda},R\sqrt{\lambda}]$ and $\lvert K\rvert\rightarrow\infty$ as $\lambda\rightarrow\infty$, synchronously. 
By monotonicity of the functions involved, we obtain uniformity of the approximation for all $u\in[-R\sqrt{\lambda},R\sqrt{\lambda}]$. Choosing $R$ large enough, we obtain a uniform estimate for all $u\in\mathbb R$ as stated in the theorem.

In order to prove the needed closeness property, we consider the random variable 
$n\mapsto \Theta(\lambda-s_2(n)+u,n)-\Phi(u/\sqrt{\lambda})n$.
By bounding the second moment, using a procedure similar to the method used by Drmota, Kauers and Spiegelhofer~\cite{DKS2016} (see also~\cite{SW2017}), we obtain an upper bound of the difference for all but few $n$.

This ``orthogonal'' approach enables us to prove a statement on the distribution of $j\mapsto \Theta(j,n)$ for most $n$ by studying the random variable $n\mapsto \Theta(j,n)$ on  
$\bigl[2^\lambda,2^{\lambda+1}\bigr)$.  

\section{Proof of the main theorem}
\subsection{Reduction of the main theorem}
Assume that $\lambda,k\geq 0$ and define the random variable
\begin{align}
	\label{eq:Xlkdefspecific}
	X_{\lambda,k}: n\mapsto \widetilde\Theta(k,n)-\Phi\left(\frac {k-\lambda}{\sqrt{\lambda}}\right)n
\end{align}
on $\bigl\{2^\lambda,2^\lambda+1,\ldots,2^{\lambda+1}-1\bigr\}$,
where $\Phi$ is the normal distribution function.
We will also use $X_{\lambda,x}$ for real values of $x$.
The central statement of this paper is contained in the following proposition, from which we will derive the main theorem.
\begin{proposition}\label{prp_closeness}
Let $R>0$ be a real number.
There exists a constant $C$ such that
\[\E (X_{\lambda,\lambda+u}^2) = \frac 1{2^\lambda}\sum_{2^\lambda\leq n<2^{\lambda+1}}X_{\lambda,\lambda+u}^2\leq C\frac{4^\lambda}{\sqrt{\lambda}}\]
for all $u$ and $\lambda$ such that $\lvert u\rvert\leq R\sqrt{\lambda}$.
In particular, for all $\varepsilon>0$ and $\lvert u\rvert\leq R\sqrt{\lambda}$ we have
\begin{equation}\label{eqn_closeness}
\left \lvert\left\{ n\in\bigl[2^\lambda,2^{\lambda+1}\bigr):
\left \lvert \widetilde\Theta(\lambda+u,n)-\Phi\left(\frac{u}{\sqrt{\lambda}}\right)n\right \rvert
\geq 2^\lambda\varepsilon\right\}\right\rvert
\leq 2^\lambda\frac C{\sqrt{\lambda} \varepsilon^2}.
\end{equation}
\end{proposition}
The second part of this proposition can be derived as follows:
We let $M$ denote the left hand side of~\eqref{eqn_closeness}.
Then we have
\[
C\frac{4^\lambda}{\sqrt{\lambda}}
\geq
\frac 1{2^\lambda}\sum_{2^\lambda\leq n<2^{\lambda+1}}X_{\lambda,\lambda+u}^2
\geq \frac 1{2^\lambda}M(2^\lambda\varepsilon)^2.\]
Note that this is similar to an application of Chebyshev's inequality. 

We wish to derive Theorem~\ref{thm_main} from this proposition.
Equation~\eqref{eqn_closeness} states that $\widetilde\Theta(\lambda+u,n)$ is ``usually'' close to a Gaussian distribution. We make this more precise in the following.
We consider the quantity $X_{\lambda,\lambda+u}$ at $N$ many points, where
$N$ is chosen later.
Set \[U=\bigl\{\left(-1+2n/N\right)R\sqrt{\lambda}:0\leq n\leq N\bigr\}.\]
Clearly, we have  
\begin{equation}\label{eqn_pivot_points}
\left\lvert\left\{n\in \bigl[2^\lambda,2^{\lambda+1}\bigr):\lvert X_{\lambda,\lambda+u}\rvert\geq 2^\lambda\varepsilon\text{ for some }u\in U\right\}\right\rvert\leq C(N+1)2^\lambda\frac 1{\sqrt{\lambda}\varepsilon^2}.
\end{equation}
Let $u_1$ and $u_2$ be adjacent elements of $U$ and assume that $u_1\leq u\leq u_2$.
By the triangle inequality, the mean value theorem and monotonicity of $\widetilde\Theta$ and $\Phi$ applied to~\eqref{eq:Xlkdefspecific} we have
\begin{align*}
\lvert X_{\lambda,\lambda+u}\rvert&\leq
\min\bigl(\lvert X_{\lambda,\lambda+u_1}\rvert,
\lvert X_{\lambda,\lambda+u_2}\rvert\bigr)
+n\bigl(\Phi(u_2/\sqrt{\lambda})-\Phi(u_1/\sqrt{\lambda})\bigr)
\\&\leq
\min\bigl(\lvert X_{\lambda,\lambda+u_1}\rvert,
\lvert X_{\lambda,\lambda+u_2}\rvert\bigr)
+2^\lambda\frac{4R}{N}.
\end{align*}

Using this inequality and~\eqref{eqn_pivot_points} it follows that
\begin{align*}
&\bigl \lvert\bigl\{n\in \bigl[2^\lambda,2^{\lambda+1}\bigr):\lvert X_{\lambda,\lambda+u}\rvert\geq 2^\lambda\bigl(\delta+4R/N\bigr)\text{ for some }u, \lvert u\rvert\leq R\sqrt{\lambda}\bigr\}\bigr \rvert
\\&\leq
\bigl \lvert\bigl\{n\in \bigl[2^\lambda,2^{\lambda+1}\bigr):
\min\bigl(\lvert X_{\lambda,\lambda+u_1}\rvert,\lvert X_{\lambda,\lambda+u_2}\rvert\bigr)\geq 2^\lambda\delta\text{ for some }u_1,u_2\in U\bigr\}\bigr \rvert
\\&=
\bigl \lvert\bigl\{n\in \bigl[2^\lambda,2^{\lambda+1}\bigr):
\lvert X_{\lambda,\lambda+u}\rvert\geq 2^\lambda\delta\text{ for some }u\in U\bigr\}\bigr \rvert
\\&\leq C(N+1)2^\lambda\frac 1{\sqrt{\lambda}\delta^2}.
\end{align*}
For given $\varepsilon>0$, we choose $\delta=\varepsilon/10$ and $N=\lceil 10R/\varepsilon\rceil$.
This implies (we also replace $n$ by $n+1$, introducing a small error which is accounted for by the error term)
\begin{equation}\label{eqn_penultimate}
\left \lvert\left\{n\in I_\lambda:
\left \lvert
\widetilde\Theta(\lambda+u,n)-\Phi\left(\frac u{\sqrt{\lambda}}\right)(n+1)\right \rvert\geq 2^\lambda\frac{\varepsilon}2\text{ for some }\lvert u\rvert\leq R\sqrt{\lambda}\right\}\right \rvert
=\LandauO\left(\frac {2^\lambda}{\sqrt{\lambda}}\right)
\end{equation}
for some implied constant depending on $\varepsilon$ and $R$.

Let $\varepsilon$ be given and choose $R$ in such a way that $\Phi(u/\sqrt{\lambda})\leq \varepsilon/2$ for $u\leq -R\sqrt{\lambda}$ 
(note that also $\Phi(u/\sqrt{\lambda})\geq 1-\varepsilon/2$ for $u\geq R\sqrt{\lambda}$).
Assume that $n\in I_\lambda$ is such that for all $\lvert u\rvert \leq R\sqrt{\lambda}$
\[
\left\lvert \frac{\widetilde\Theta(\lambda+u,n)}{n+1}-\Phi\left(\frac u{\sqrt{\lambda}}\right)\right\rvert< \frac \varepsilon 2.
\]

Then by monotonicity (remember that $\Theta(k,n)$ are partial sums) we have $\widetilde\Theta(\lambda+u,n)/(n+1)\in[0,\varepsilon]$, $\Phi(u/\sqrt{\lambda})\in[0,\varepsilon]$ for $u\leq -R\sqrt{\lambda}$
and likewise
$\widetilde\Theta(\lambda+u,n)/(n+1)\in[1-\varepsilon,1]$, $\Phi(u/\sqrt{\lambda})\in[1-\varepsilon,1]$ for $u\geq R\sqrt{\lambda}$, therefore
\[
\left\lvert \frac{\widetilde\Theta(\lambda+u,n)}{n+1}-\Phi\left(\frac u{\sqrt{\lambda}}\right)\right\rvert<\varepsilon
\]
for all real $u$.
Using also equation~\eqref{eqn_penultimate}
we obtain
\begin{align*}
&\left \lvert\left\{n\in I_\lambda:
\left \lvert
\frac{\widetilde\Theta(\lambda+u,n)}{n+1}-\Phi\left(\frac u{\sqrt{\lambda}}\right)\right \rvert\geq \varepsilon \text{ for some } u\in\mathbb R\right\}\right \rvert
\\&\leq
\left \lvert\left\{n\in I_\lambda:
\left \lvert
\frac{\widetilde\Theta(\lambda+u,n)}{n+1}-\Phi\left(\frac u{\sqrt{\lambda}}\right)\right \rvert\geq \frac{\varepsilon}2\text{ for some }\lvert u\rvert\leq R\sqrt{\lambda}\right\}\right \rvert
\\&\leq
\left \lvert\left\{n\in I_\lambda:
\left \lvert
\widetilde\Theta(\lambda+u,n)-\Phi\left(\frac u{\sqrt{\lambda}}\right)(n+1)\right \rvert\geq 2^\lambda\frac{\varepsilon}2\text{ for some }\lvert u\rvert\leq R\sqrt{\lambda}\right\}\right \rvert
\\&=\LandauO\left(\frac {2^\lambda}{\sqrt{\lambda}}\right)
\end{align*}
and the proof of the second part is complete.

It remains to prove the first part of Proposition~\ref{prp_closeness}. Motivated by Chebyshev's inequality, the main idea in its proof is to show that the the random variable $X_{\lambda,k}$ possesses a small second moment.
\subsection{A recurrence relation for the values $\vartheta_p(j,n)$}\label{sec_rec}
Carlitz~\cite{C1967} found a recurrence for the values $\vartheta_p(j,n)$,
involving a second family $\psi_p$ of values.
We will be working with a shifted and rarefied family $\tilde\vartheta_p$, which satisfies a simpler recurrence relation
(compare the paper by the authors~\cite[Section 2.3]{SW2016}):
define, for $k,n\geq 0$,
\begin{align}\label{eqn:def_tilde_theta}
\tilde\vartheta_p(k,n) &=
\begin{cases}
  \vartheta_p\Bigl(\frac{k-s_p(n)}{p-1},n\Bigr), & k\geq s_p(n)\text{ and }p-1\mid k-s_p(n);\\
  0,                     & \text{otherwise}.
\end{cases}
\end{align}
\def\extrarowheight{1ex}             
\newcommand{\colwidth}{1em}
\newcommand{\SetWidth}[1]{\makebox[\widthof{\colwidth}]{$#1$}}%
\begin{table}[h!]
\[
\!\!\!
\begin{array}{c|c@{}c@{}c@{}c@{}c@{}c@{}c@{}c@{}c@{}c@{}c@{}c@{}c@{}c@{}c@{}c@{}c@{}c}
&\SetWidth{0}&\SetWidth{1}&\SetWidth{2}&\SetWidth{3}&\SetWidth{4}&\SetWidth{5}&\SetWidth{6}&\SetWidth{7}&\SetWidth{8}&\SetWidth{9}&\SetWidth{10}&\SetWidth{11}&\SetWidth{12}&\SetWidth{13}&\SetWidth{14}&\SetWidth{15}&\SetWidth{16}&\SetWidth{17}\\
\hline
 0& 1&  &  &  &  &  &  &  &  &  &  &  &  &  &  &  &  &\\
 1&  & 2& 2&  & 2&  &  &  & 2&  &  &  &  &  &  &  & 2&\\
 2&  &  & 1& 4& 1& 4& 4&  & 1& 4& 4&  & 4&  &  &  & 1& 4\\
 3&  &  &  &  & 2& 2& 2& 8& 2& 2& 4& 8& 2& 8& 8&  & 2& 2\\
 4&  &  &  &  &  &  & 1&  & 4& 4& 1& 4& 5& 4& 4&16& 4& 4\\
 5&  &  &  &  &  &  &  &  &  &  & 2&  & 2& 2& 2&  & 8& 8\\
 6&  &  &  &  &  &  &  &  &  &  &  &  &  &  & 1&  &  &
\end{array}
\]
\caption{Some coefficients of $\tilde\vartheta_2(k,n)$. The first variable corresponds to the row number.}
\label{tbl:tilde_theta_values2}
\end{table}
\def\extrarowheight{0ex}             
Setting for simplicity $\tilde\vartheta_p(k,n)=0$ if $k<0$ or $n<0$,
we obtain the following recurrence relation for $k,n\geq 0$.
\begin{equation}\label{eqn:tilde_theta_rec1}
\begin{aligned}
\tilde\vartheta_p(0,n)&=\delta_{0,n},&\text{for }n\geq 0;\\
\tilde\vartheta_p(k,0)&=\delta_{k,0},&\text{for }k\geq 0,\end{aligned}
\end{equation}
and for $n\geq 0$ and $0\leq a<p$,
\begin{equation}\label{eqn:tilde_theta_rec2}
\tilde\vartheta_p(k,pn+a)=(a+1)\tilde\vartheta_p(k-a,n)+(p-a-1)\tilde\vartheta_p(k-p-a,n-1).
\end{equation}

For the rest of the paper, we will restrict ourselves to the case $p=2$.
We will therefore usually omit the subscript $2$.
Assume that $\lambda,k\geq 0$.
Then the above recurrence reads
\begin{equation}\label{eqn:tilde_theta_rec3}
\begin{aligned}
\tilde\vartheta(0,n)&=\delta_{0,n},&\text{for }n\geq 0;\\
\tilde\vartheta(k,0)&=\delta_{k,0},&\text{for }k\geq 0;\\
\tilde\vartheta(k,2n)&=\tilde\vartheta(k,n)+\tilde\vartheta_2(k-2,n-1),\\
\tilde\vartheta(k,2n+1)&=2\tilde\vartheta(k-1,n).
\end{aligned}
\end{equation}

Note that
\begin{equation}\label{eqn:vartheta_Kronecker}
\tilde\vartheta(k,2^\lambda-1)=2^\lambda\delta_{k,\lambda}.
\end{equation}
\subsection{OGFs for the moments}
We are interested in the quantity
\begin{align}
\label{eqn:Theta}
\widetilde\Theta(j,n)=\sum_{i\leq j}\tilde\vartheta(i,n).
\end{align}
For most $n$ these values should follow a normal distribution with expectation $\lambda$ and variance $\lambda$.

In order to compute the second moment of $X_{\lambda,k}$, we will have to treat the following quantities.
We define
\begin{align}
m_{\lambda,k}&\coloneqq
\frac 1{2^\lambda}
\sum_{2^\lambda\leq n<2^{\lambda+1}}
\tilde\vartheta(k,n),&
\mathfrak m_{\lambda,k}&\coloneqq
\frac 1{2^\lambda}
\sum_{2^\lambda\leq n<2^{\lambda+1}}
\widetilde\Theta(k,n),&
\\m'_{\lambda,k}&\coloneqq
\frac 1{2^\lambda}
\sum_{2^\lambda\leq n<2^{\lambda+1}}
n\,\tilde\vartheta(k,n),&
\mathfrak m'_{\lambda,k}&\coloneqq
\frac 1{2^\lambda}
\sum_{2^\lambda\leq n<2^{\lambda+1}}
n\,\widetilde\Theta(k,n),\\
m^{(2)}_{\lambda,k}&\coloneqq
\frac 1{2^\lambda}
\sum_{2^\lambda\leq n<2^{\lambda+1}}
\tilde\vartheta(k,n)^2&
\mathfrak m^{(2)}_{\lambda,k}&\coloneqq
\frac 1{2^\lambda}
\sum_{2^\lambda\leq n<2^{\lambda+1}}
\widetilde\Theta(k,n)^2
\end{align}
and the corresponding generating functions
\begin{align*}
M(x,y)&=\sum_{\lambda,k\geq 0}m_{\lambda,k}x^\lambda y^k,&
\mathfrak M(x,y)&=\sum_{\lambda,k\geq 0}\mathfrak m_{\lambda,k}x^\lambda y^k,&
\\
M'(x,y)&=\sum_{\lambda,k\geq 0}m'_{\lambda,k}x^\lambda y^k,&
\mathfrak M'(x,y)&=\sum_{\lambda,k\geq 0}\mathfrak m'_{\lambda,k}x^\lambda y^k,
\\
M^{(2)}(x,y)&=\sum_{\lambda,k\geq 0}m^{(2)}_{\lambda,k}x^\lambda y^k,&
\mathfrak M^{(2)}(x,y)&=\sum_{\lambda,k\geq 0}\mathfrak m^{(2)}_{\lambda,k}x^\lambda y^k.
\end{align*}
We begin with the easier cases, concerning $m_{\lambda,k}$ and $m'_{\lambda,k}$.
We have $m_{0,k}=2\delta_{k,1}$ and for $k\geq 0$ and $\lambda\geq 1$ we obtain by splitting into even and odd integers,
applying the recurrence~\eqref{eqn:tilde_theta_rec3} and the identity $\tilde\vartheta(k,2^\lambda-1)=2^\lambda\delta_{k,\lambda}$,
\begin{equation}\label{eqn:m_rec_one_step}\nonumber
\begin{aligned}
m_{\lambda,k}
&=\frac 1{2^\lambda}\sum_{2^{\lambda-1}\leq n<2^\lambda}\tilde\vartheta(k,2n)
+\frac 1{2^\lambda}\sum_{2^{\lambda-1}\leq n<2^\lambda}\tilde\vartheta(k,2n+1)\\
&=\frac 12\frac 1{2^{\lambda-1}}\sum_{2^{\lambda-1}\leq n<2^\lambda}\tilde\vartheta(k,n)
+\frac 12\frac 1{2^{\lambda-1}}\sum_{2^{\lambda-1}\leq n<2^\lambda}\tilde\vartheta(k-2,n-1)\\
&+\frac 1{2^{\lambda-1}}\sum_{2^{\lambda-1}\leq n<2^\lambda}\tilde\vartheta(k-1,n)
\\&=\frac 12\bigl(m_{\lambda-1,k-2}+2m_{\lambda-1,k-1}+m_{\lambda-1,k}\bigr)
+\frac 1{2^\lambda}\tilde\vartheta(k-2,2^{\lambda-1}-1)
-\frac 1{2^\lambda}\tilde\vartheta(k-2,2^{\lambda}-1)\\
\\&=\frac 12\bigl(m_{\lambda-1,k-2}+2m_{\lambda-1,k-1}+m_{\lambda-1,k}\bigr)
+\frac 12\delta_{\lambda,k-1}-\delta_{\lambda,k-2}.
\end{aligned}
\end{equation}
For convenience, as we noted above, we set $\tilde\vartheta(k,n)=0$ if $k<0$ or $n<0$.
We obtain
\begin{align*}
M(x,y)&=\sum_{\lambda,k\geq 0}m_{\lambda,k}x^\lambda y^k
=\sum_{k\geq 0}m_{0,k}y^k
+
\sum_{\substack{\lambda\geq 1\\k\geq 0}}m_{\lambda,k}x^\lambda y^k\\
&=2y
+
\sum_{\substack{\lambda\geq 1\\k\geq 0}}
\Bigl(
\frac 12 m_{\lambda-1,k-2}+m_{\lambda-1,k-1}+\frac 12 m_{\lambda-1,k}
+\frac 12\delta_{\lambda,k-1}-\delta_{\lambda,k-2}
\Bigr)x^\lambda y^k\\
&=2y+\biggl(\frac {xy^2}2+xy+\frac x2\biggr)M(x,y)+
\frac 12\sum_{\lambda\geq 1}x^\lambda y^{\lambda+1}
-\sum_{\lambda\geq 1}x^\lambda y^{\lambda+2}\\
&=2y+\frac x2(1+y)^2M(x,y)+\frac y2\biggl(\frac 1{1-xy}-1\biggr)
-y^2\biggl(\frac 1{1-xy}-1\biggr)\\
&=2y+\frac x2(1+y)^2M(x,y)+
\frac 12\frac {xy^2}{1-xy}\bigl(1-2y\bigr),
\end{align*}
therefore
\begin{equation}\label{eqn:M_identity}
M(x,y)=\frac{2y+\frac 12\frac{xy^2}{1-xy}(1-2y)}{1-\frac x2(1+y)^2}
=\frac y2\frac{4-3xy-2xy^2}{(1-xy)\bigl(1-\frac x2(1+y)^2\bigr)}.
\end{equation}
We will use this identity in a moment in the treatment of the values $m'_{\lambda,k}$.
Moreover, we clearly have
\begin{equation}\label{eqn:frakM_identity}
\mathfrak M(x,y)=\frac 1{1-y}\frac y2\frac{4-3xy-2xy^2}{\bigl(1-xy\bigr)\bigl(1-\frac x2(1+y)^2\bigr)}.
\end{equation}
For $\lambda\geq 1$ and $k\geq 0$ we have
\begin{align*}
m'_{\lambda,k}
&=\frac 1{2^\lambda}\sum_{2^{\lambda-1}\leq n<2^\lambda}2n\tilde\vartheta(k,2n)
+\frac 1{2^\lambda}\sum_{2^{\lambda-1}\leq n<2^\lambda}(2n+1)\tilde\vartheta(k,2n+1)\\
&=\frac 1{2^{\lambda-1}}\sum_{2^{\lambda-1}\leq n<2^\lambda}n\tilde\vartheta(k,n)
+\frac 1{2^{\lambda-1}}\sum_{2^{\lambda-1}-1\leq n<2^\lambda-1}(n+1)\tilde\vartheta(k-2,n)\\
&+\frac 2{2^{\lambda-1}}\sum_{2^{\lambda-1}\leq n<2^\lambda}n\tilde\vartheta(k-1,n)+\frac 1{2^{\lambda-1}}\sum_{2^{\lambda-1}\leq n<2^\lambda}\tilde\vartheta(k-1,n)
\\&=m'_{\lambda-1,k}+m'_{\lambda-1,k-2}
+\frac 1{2^{\lambda-1}}\bigl(2^{\lambda-1}-1\bigr)\tilde\vartheta\bigl(k-2,2^{\lambda-1}-1\bigr)\\
&-\frac 1{2^{\lambda-1}}\bigl(2^{\lambda}-1\bigr)\tilde\vartheta\bigl(k-2,2^{\lambda}-1\bigr)+m_{\lambda-1,k-2}
+\frac 1{2^{\lambda-1}}\tilde\vartheta\bigl(k-2,2^{\lambda-1}-1\bigr)\\
&-\frac 1{2^{\lambda-1}}\tilde\vartheta\bigl(k-2,2^{\lambda}-1\bigr)
+2m'_{\lambda-1,k-1}+m_{\lambda-1,k-1}
\\&=m'_{\lambda-1,k}+2m'_{\lambda-1,k-1}+m'_{\lambda-1,k-2}+m_{\lambda-1,k-1}+m_{\lambda-1,k-2}
\\&+2^{\lambda-1}\delta_{\lambda,k-1}
-2^{\lambda+1}\delta_{\lambda,k-2}
\end{align*}
We obtain
\begin{align*}
M'(x,y)&=\sum_{\lambda,k\geq 0}m'_{\lambda,k}x^\lambda y^k
=\sum_{k\geq 0}m'_{0,k}y^k
+
\sum_{\substack{\lambda\geq 1\\k\geq 0}}m'_{\lambda,k}x^\lambda y^k\\
&=
2y
+
\sum_{\substack{\lambda\geq 1\\k\geq 0}}
\Bigl(m'_{\lambda-1,k}+2m'_{\lambda-1,k-1}+m'_{\lambda-1,k-2}\\&+m_{\lambda-1,k-1}+m_{\lambda-1,k-2}
+2^{\lambda-1}\delta_{\lambda,k-1}
-2^{\lambda+1}\delta_{\lambda,k-2}
\Bigr)x^\lambda y^k\\
&=2y+\bigl(x+2xy+xy^2\bigr)M'(x,y)
+xy(1+y)M(x,y)\\
&+\sum_{\lambda\geq 0}2^\lambda x^{\lambda+1} y^{\lambda+2}
-\sum_{\lambda\geq 0}2^{\lambda+2}x^{\lambda+1} y^{\lambda+3}\\
&=2y+x(1+y)^2M'(x,y)+xy(1+y)M(x,y)
+\frac {xy^2(1-4y)}{1-2xy}.
\end{align*}
Inserting the formula for $M$, we obtain
\begin{equation}\label{eqn:Mprime_identity}
M'(x,y)=
\frac 1{1-x(1+y)^2}
\biggl(
2y+
\frac {xy^2(1+y)}2\frac{4-3xy-2xy^2}{(1-xy)\bigl(1-\frac x2(1+y)^2\bigr)}
+\frac {xy^2(1-4y)}{1-2xy}
\biggr).
\end{equation}
and
\begin{equation}\label{eqn:frakMprime_identity}
\mathfrak M'(x,y)=
\frac 1{\bigl(1-y\bigr)\bigl(1-x(1+y)^2\bigr)}
\biggl(
2y+
\frac {xy^2(1+y)}2\frac{4-3xy-2xy^2}{(1-xy)\bigl(1-\frac x2(1+y)^2\bigr)}
+\frac {xy^2(1-4y)}{1-2xy}
\biggr).
\end{equation}
Note that the denominator of $M'$ has a simple structure, therefore it will not be too difficult to analyze the coefficients asymptotically.

Let us proceed to the main term.
We want to extract this term as a diagonal of a trivariate generating function.
We define therefore, for $\lambda,k,\ell\geq 0$,
\begin{equation}\label{eqn:abc_def}
\begin{aligned}
a_{\lambda,k,\ell}&\coloneqq
\sum_{2^\lambda\leq n<2^{\lambda+1}}
\tilde\vartheta(k,n)\tilde\vartheta(\ell,n),\\
b_{\lambda,k,\ell}&\coloneqq
\sum_{2^\lambda\leq n<2^{\lambda+1}}
\tilde\vartheta(k,n)\tilde\vartheta(\ell,n-1),\\
c_{\lambda,k,\ell}&\coloneqq
\sum_{2^\lambda\leq n<2^{\lambda+1}}
\tilde\vartheta(k,n-1)\tilde\vartheta(\ell,n),
\end{aligned}
\end{equation}
where $b$ and $c$ will act as auxiliary variables.
For convenience, we define $a_{\lambda,k,\ell}=b_{\lambda,k,\ell}=c_{\lambda,k,\ell}=0$ for $k<0$ or $\ell<0$.

By splitting into even and odd indices and using~\eqref{eqn:tilde_theta_rec3},
we obtain for $\lambda\geq 1$ and $k,\ell\geq 0$
\begin{align*}
a_{\lambda,k,\ell}&=
\sum_{2^{\lambda-1}\leq n<2^\lambda}
\bigl(\tilde\vartheta(k,n)+\tilde\vartheta(k-2,n-1)\bigr)
\bigl(\tilde\vartheta(\ell,n)+\tilde\vartheta(\ell-2,n-1)\bigr)
\\&+4\sum_{2^{\lambda-1}\leq n<2^\lambda}
\tilde\vartheta(k-1,n)\,\tilde\vartheta(\ell-1,n)\\
&=a_{\lambda-1,k,\ell}+b_{\lambda-1,k,\ell-2}+c_{\lambda-1,k-2,\ell}+a_{\lambda-1,k-2,\ell-2}+4a_{\lambda-1,k-1,\ell-1}\\
&+\tilde\vartheta(k-2,2^{\lambda-1}-1)\,\tilde\vartheta(\ell-2,2^{\lambda-1}-1)
-\tilde\vartheta(k-2,2^{\lambda}-1)\,\tilde\vartheta(\ell-2,2^{\lambda}-1)\\
&=a_{\lambda-1,k,\ell}+b_{\lambda-1,k,\ell-2}+c_{\lambda-1,k-2,\ell}+a_{\lambda-1,k-2,\ell-2}+4a_{\lambda-1,k-1,\ell-1}\\
&+4^{\lambda-1}\delta_{k-1,\lambda}\delta_{\ell-1,\lambda}
-4^\lambda\delta_{k-2,\lambda}\delta_{\ell-2,\lambda},
\end{align*}
\begin{align*}
b_{\lambda,k,\ell}&=
2\sum_{2^{\lambda-1}\leq n<2^\lambda}
\bigl(\tilde\vartheta(k,n)+\tilde\vartheta(k-2,n-1)\bigr)\,\tilde\vartheta(\ell-1,n-1)
\\&+2\sum_{2^{\lambda-1}\leq n<2^\lambda}
\tilde\vartheta(k-1,n)
\,
\bigl(\tilde\vartheta(\ell,n)+\tilde\vartheta(\ell-2,n-1)\bigr)\\
&=2b_{\lambda-1,k,\ell-1}+2a_{\lambda-1,k-2,\ell-1}+2a_{\lambda-1,k-1,\ell}+2b_{\lambda-1,k-1,\ell-2}\\
&+2\tilde\vartheta(k-2,2^{\lambda-1}-1)\tilde\vartheta(\ell-1,2^{\lambda-1}-1)-2\tilde\vartheta(k-2,2^{\lambda}-1)\tilde\vartheta(\ell-1,2^{\lambda}-1)\\
&=2b_{\lambda-1,k,\ell-1}+2a_{\lambda-1,k-2,\ell-1}+2a_{\lambda-1,k-1,\ell}+2b_{\lambda-1,k-1,\ell-2}\\
&+2\cdot 4^{\lambda-1}\delta_{k-1,\lambda}\delta_{\ell,\lambda}
-2\cdot 4^\lambda\delta_{k-2,\lambda}\delta_{\ell-1,\lambda},
\end{align*}
and by exchanging $b$ and $c$ resp. $k$ and $\ell$,
\begin{align*}
c_{\lambda,k,\ell}
&=2c_{\lambda-1,k-1,\ell}+2a_{\lambda-1,k-1,\ell-2}+2a_{\lambda-1,k,\ell-1}+2c_{\lambda-1,k-2,\ell-1}\\
&+2\cdot 4^{\lambda-1}\delta_{k,\lambda}\delta_{\ell-1,\lambda}
-2\cdot 4^\lambda\delta_{k-1,\lambda}\delta_{\ell-2,\lambda}.
\end{align*}
We translate these recurrences into identities for trivariate generating functions.
Set
\begin{align*}
A(x,y,z)&\coloneqq
\sum_{\lambda,k,\ell\geq 0}a_{\lambda,k,\ell}x^\lambda y^k z^\ell,
\\
B(x,y,z)&\coloneqq
\sum_{\lambda,k,\ell\geq 0}b_{\lambda,k,\ell}x^\lambda y^k z^\ell,\\
C(x,y,z)&\coloneqq
\sum_{\lambda,k,\ell\geq 0}c_{\lambda,k,\ell}x^\lambda y^k z^\ell.
\end{align*}
Then
\begin{align*}
A(x,y,z)&=
\sum_{k,\ell\geq 0}
\tilde\vartheta(k,1)\tilde\vartheta(\ell,1)
y^kz^\ell
\\&+\sum_{\substack{\lambda\geq 1\\k,\ell\geq 0}}
\Bigl(a_{\lambda-1,k,\ell}+b_{\lambda-1,k,\ell-2}+c_{\lambda-1,k-2,\ell}+a_{\lambda-1,k-2,\ell-2}+4a_{\lambda-1,k-1,\ell-1}\Bigr)
x^\lambda y^kz^\ell
\\&+\sum_{\substack{\lambda\geq 1\\k,\ell\geq 0}}
4^{\lambda-1}\delta_{k-1,\lambda}\delta_{\ell-1,\lambda}
x^\lambda y^kz^\ell
-
\sum_{\substack{\lambda\geq 1\\k,\ell\geq 0}}
4^\lambda\delta_{k-2,\lambda}\delta_{\ell-2,\lambda}
x^\lambda y^kz^\ell
\\
&=4yz+x\bigl(1+4yz+y^2z^2\bigr)A(x,y,z)
+xz^2B(x,y,z)+xy^2C(x,y,z)\\
&+
\sum_{\lambda\geq 1}4^{\lambda-1}x^\lambda y^{\lambda+1}z^{\lambda+1}
-4\sum_{\lambda\geq 1}4^{\lambda-1}x^\lambda y^{\lambda+2}z^{\lambda+2}\\
&=4yz+x\bigl(1+4yz+y^2z^2\bigr)A(x,y,z)
+xz^2B(x,y,z)+xy^2C(x,y,z)\\
&+
xy^2z^2\frac{1-4yz}{1-4xyz},
\\[\baselineskip]
B(x,y,z)&=
\sum_{k,\ell\geq 0}
\tilde\vartheta(k,1)\tilde\vartheta(\ell,0)
y^kz^\ell
\\&+2\sum_{\substack{\lambda\geq 1\\k,\ell\geq 0}}
\Bigl(
b_{\lambda-1,k,\ell-1}+a_{\lambda-1,k-2,\ell-1}+a_{\lambda-1,k-1,\ell}+b_{\lambda-1,k-1,\ell-2}
\Bigr)
x^\lambda y^kz^\ell
\\&+2\sum_{\substack{\lambda\geq 1\\k,\ell\geq 0}}
4^{\lambda-1}\delta_{k-1,\lambda}\delta_{\ell,\lambda}
x^\lambda y^kz^\ell
-2\sum_{\substack{\lambda\geq 1\\k,\ell\geq 0}}
4^\lambda\delta_{k-2,\lambda}\delta_{\ell-1,\lambda}
x^\lambda y^kz^\ell
\\
&=2y+2xy\bigl(1+yz\bigr)A(x,y,z)
+2xz\bigl(1+yz\bigr)B(x,y,z)\\
&+2\sum_{\lambda\geq 1}4^{\lambda-1}x^\lambda y^{\lambda+1}z^\lambda
-8\sum_{\lambda\geq 1}4^{\lambda-1}x^\lambda y^{\lambda+2}z^{\lambda+1}\\
&=2y+2xy\bigl(1+yz\bigr)A(x,y,z)
+xz\bigl(1+yz\bigr)B(x,y,z)\\
&+2xy^2z
\frac{1-4yz}{1-4xyz},
\\[\baselineskip]
C(x,y,z)
&=2z+2xz\bigl(1+yz\bigr)A(x,y,z)
+2xy\bigl(1+yz\bigr)C(x,y,z)\\
&+2xyz^2
\frac{1-4yz}{1-4xyz}.
\end{align*}
It follows that
\begin{align*}
B(x,y,z)&=\frac{2y+2xy^2z\frac{1-4yz}{1-4xyz}}{1-2xz(1+yz)}
+\frac{2xy(1+yz)}{1-2xz(1+yz)}A(x,y,z),\\
C(x,y,z)&=\frac{2z+2xyz^2\frac{1-4yz}{1-4xyz}}{1-2xy(1+yz)}
+\frac{2xz(1+yz)}{1-2xy(1+yz)}A(x,y,z),
\end{align*}
and therefore we obtain after some rewriting
\begin{align}\label{eqn:M2_identity}
\begin{aligned}
\mathfrak M^{(2)}&=
\frac 1{2^\lambda (1-y)(1-z)}A(x,y,z)\\&=
\frac 1{2^\lambda (1-y)(1-z)}
\frac{4yz
+xz^2\frac{2y+2xy^2z\frac{1-4yz}{1-4xyz}}{1-2xz(1+yz)}
+xy^2\frac{2z+2xyz^2\frac{1-4yz}{1-4xyz}}{1-2xy(1+yz)}
+xy^2z^2\frac{1-4yz}{1-4xyz}}{1-x(1+yz)^2
-\frac{xyz}{1-2xz(1+yz)}
-\frac{xyz}{1-2xy(1+yz)}}.
\end{aligned}
\end{align}
Note that we have the same denominator as in~\cite{DKS2016}.
\subsection{Asymptotic expansion of the first moment}
Recall that we want to compute ${\mathfrak m}_{\lambda,\lambda+u}=[x^\lambda y^{\lambda+u}] \mathfrak M(x,y)$ where the rational function $\mathfrak M(x,y)$ is given in~\eqref{eqn:frakM_identity}. We will adapt the method of~\cite{D1994} which also captures the (Gaussian) fluctuations $n+u$ for $u$ sufficiently small.
\begin{lemma}
	\label{lem:asyfirstmom}
	For $\lambda \to \infty$ we have
	\begin{align*}
		{\mathfrak m}_{\lambda,\lambda+u} &= \frac{3}{2} 2^\lambda \Phi\left(\frac{u}{\sqrt{\lambda}}\right) \left(1 + \LandauO\left(\frac{1}{\sqrt{\lambda}}\right) + \LandauO\left(\frac{u^2}{\sqrt{\lambda^{3} } }\right) \right).
	\end{align*}
\end{lemma}

\begin{proof}
The nature of the random variable $\widetilde\Theta$ displayed in~\eqref{eqn:Theta} implies that partial sums will play a key role. Their avatar is encoded in the factor $\frac{1}{1-y}$ of $\mathfrak M(x,y)$. It proves convenient, to first compute the asymptotic expansion of the coefficients of 
\begin{align*}
	\widetilde{\mathfrak M}(x,y) \coloneqq (1-y)  \mathfrak M(x,y).
\end{align*}

First, we extract the $n$-th coefficient with respect to $x$ by a partial fraction decomposition. We get
\begin{align*}
	  \widetilde{\mathfrak m}_{\lambda,\lambda+u} \coloneqq [x^\lambda y^{\lambda+u}] \widetilde{\mathfrak M}(x,y) &= \frac{1}{2^{\lambda+1}}[y^{\lambda+u}] \frac{y(y+2)(1+y)^{2\lambda}}{1+y^2} + \LandauO(1).
\end{align*} 
The error term arises from the second fraction in the partial fraction decomposition. Next we apply a simple generalization of~\cite[Theorem~2]{D1994}. We omit the details as the ideas carry directly over. In words, we apply a saddle point method like extensively discussed in~\cite{FS2009}. The dominant singularity $\rho_\lambda$ and the variation constant $\sigma_\lambda^2$ are given by
\begin{align*}
	\rho_\lambda &= 1 - \frac{2}{3\lambda} + \LandauO\left(\frac{1}{\lambda^2}\right), \\
	\sigma_\lambda^2 &= \frac{\lambda}{2} - \frac{7}{9} + \LandauO\left(\frac{1}{\lambda}\right).
\end{align*}
We get
\begin{align*}
	\widetilde{\mathfrak m}_{\lambda,\lambda+u} &= \frac{3}{2} \frac{2^\lambda}{\sqrt{\pi \lambda}} e^{-\frac{u^2}{\lambda}} \left( 1 + \LandauO\left(\frac{u}{\lambda}\right)\right).
\end{align*}
Next, we come back to ${\mathfrak M}(x,y)$. 
We want to compute
\begin{align*}
	[x^\lambda y^{\lambda+u}] \mathfrak M(x,y) &= \sum_{k=0}^{\lambda + u} \widetilde{\mathfrak m}_{\lambda,k}.
\end{align*}
We choose a positive number $\varepsilon > \frac{1}{2}$.
Then, we split the sum at $u = \lambda^{\varepsilon}$ into two parts. Using our asymptotic result for $\widetilde{\mathfrak m}_{\lambda,\lambda+u}$, we see that for large $\lambda$ the first sum is negligible as
\begin{align*}
	\sum_{k=0}^{\lambda-\lambda^{\varepsilon}} \widetilde{\mathfrak m}_{\lambda,\lambda+u} &\leq  
		 C (\lambda - \lambda^{\varepsilon}) \frac{2^{\lambda}}{\sqrt{\lambda}} e^{-\lambda^{2\varepsilon-1}} = 
		\Landauo\left(e^{-\lambda^\delta}\right),
\end{align*}
for $0 < \delta < 2\varepsilon-1$ and a suitable constant $C>0$. 
We continue with the second sum using our asymptotic result again. For the main term we get
\begin{align*}
	\frac{3}{2} \frac{2^\lambda}{\sqrt{\pi \lambda }} \sum_{k=-\lambda^{\varepsilon}}^{u} e^{-\frac{k^2}{\lambda}} &= \frac{3}{2} \frac{2^\lambda}{\sqrt{2\pi}} \int_{-\lambda^{\varepsilon - 1/2}}^{u/\sqrt{\lambda}} e^{-\frac{v^2}{2}} \, dv \left(1 + \LandauO\left(\frac{1}{\sqrt{\lambda}}\right)\right) \\
	&= \frac{3}{2} 2^\lambda \Phi\left(\frac{u}{\sqrt{\lambda}}\right) \left(1 + \LandauO\left(\frac{1}{\sqrt{\lambda}}\right)\right),
\end{align*}
where in the first equality we also made a change of variable scaling the path of integration by $\sqrt{\lambda}$, and in the third equality we completed the tail. This last operation only introduced an error term of order $\LandauO(e^{-n^{\varepsilon-1/2}})$.

It remains to consider the error term. Yet similar reasoning shows
\begin{align*}
	\sum_{k=-\lambda^{\varepsilon}}^{u} \frac{k}{\lambda} e^{-\frac{k^2}{\lambda}} &= 2e^{-\frac{u^2}{\lambda}} + \LandauO\left(\frac{1}{\lambda}\right).
\end{align*}
This ends the proof.
\end{proof}
\subsection{Asymptotic expansion of the mixed term}

As a next step we compute $\mathfrak m'_{\lambda,\lambda+u} = [x^\lambda y^{\lambda+u}] {\mathfrak M'}(x,y)$ from~\eqref{eqn:Mprime_identity}.

\begin{lemma}
	\label{lem:asymixedmom}
	For $\lambda \to \infty$ we have
	\begin{align*}
		\mathfrak m'_{\lambda,\lambda+u} &= \frac{7}{3} 4^\lambda \Phi\left(\frac{u}{\sqrt{\lambda}}\right) \left(1 + \LandauO\left(\frac{1}{\sqrt{\lambda}}\right) + \LandauO\left(\frac{u^2}{\sqrt{\lambda^{3} } }\right) \right).
	\end{align*}
\end{lemma}

\begin{proof}
	The ideas and techniques are the same as in the proof of Lemma~\ref{lem:asyfirstmom} and were performed with the help of Maple. As the denominator is of order $4$ in $x$, we perform a partial fraction decomposition with respect to $x$. This gives $4$ rational functions where we extract coefficient of $x^{\lambda}$ and get rational functions in $y$ where only one of them has coefficients of order $\LandauO(4^{\lambda})$. We omit these technical steps.
\end{proof}

\subsection{Asymptotic expansion of the second moment}
We want to show an asymptotic formula for the values $\mathfrak m^{(2)}_{\lambda,\lambda+u}$.

\begin{lemma}
	\label{lem:asysecondmom}
	For $\lambda \to \infty$ we have
	\begin{align*}
		\mathfrak m^{(2)}_{\lambda,\lambda+u} &= \frac{7}{3} 4^\lambda \, \Phi \left(\frac{u}{\sqrt{\lambda}}\right)^2 \left(1 + \LandauO\left(\frac{1}{\sqrt{\lambda}}\right) + \LandauO\left(\frac{u^2}{\sqrt{\lambda^{3} } }\right) \right).
	\end{align*}
\end{lemma}

\begin{proof}
We use the same idea as before performing the summation (here in two variables) later. But as we are dealing with a function in three variables, it is more suitable to use the techniques of~\cite[Proposition~4.3]{DKS2016}. We only sketch the main differences here.
(We are aware of the theory of Pemantle and Wilson~\cite{PW2013}, which provides a method to obtain asymptotics of certain multivariate functions.
However we noted in~\cite{DKS2016} that the trivariate function considered there turns out to be a limit case which has to be treated separately.
As this function possesses, up to a factor, the same denominator as our function $\mathfrak M^{(2)}$, the same restriction applies here.)
Set
\[
G(x,y,z)\coloneqq 4yz
+xz^2\frac{2y+2xy^2z\frac{1-4yz}{1-4xyz}}{1-2xz(1+yz)}
+xy^2\frac{2z+2xyz^2\frac{1-4yz}{1-4xyz}}{1-2xy(1+yz)}
+xy^2z^2\frac{1-4yz}{1-4xyz}
\]
and
\[H(x,y,z)\coloneqq 1-x(1+yz)^2
-\frac{xyz}{1-2xz(1+yz)}
-\frac{xyz}{1-2xy(1+yz)}.
\]
Then, by definition, we have
\begin{align*}
\mathfrak M^{(2)}(x,y,z)=
\frac 1{2^{\lambda}(1-y)(1-z)}\frac{G(x,y,z)}{H(x,y,z)}.
\end{align*}

The first difference is that we continue to compute the asymptotic expansion of the coefficients of $\frac{G(x,y,z)}{H(x,y,z)}$ first. As done before, we define the shorthand 
$$\widetilde{\mathfrak m}^{(2)}_{\lambda,\lambda+u_1,\lambda+u_2} \coloneqq \frac{1}{2^\lambda}[x^\lambda y^{\lambda+u_1} z^{\lambda+u_2}]  \frac{G(x,y,z)}{H(x,y,z)}.$$
Note that the denominator $H$ is identical to the denominator in~\cite{DKS2016}. Thus we can follow the known computations. Here we introduce a perturbation in the y- and z-coordinate given by $\lambda+u_1$ and $\lambda+u_2$, respectively. Then, Cauchy's integral formula over the same contour as in~\cite{DKS2016} gives

\begin{align*}
\bigl[x^{\lambda} &y^{\lambda+u} z^{\lambda+u}\bigr]\, F(x,y,z)
 = O\bigl(8^{(1-\varepsilon)\lambda}\bigr)\\
&+ \frac 1{(2\pi i)^2} \iint\limits_{\gamma_1 \times \gamma_1}
 \frac{-G(f(y,z),y,z)}{xyz H_x(f(y,z),y,z)} \bigl(f(y,z)y^{1+u_1/\lambda}z^{1+u_2/\lambda}\bigr)^{-\lambda} \mathrm dy\,\mathrm dz.
\end{align*}
where the implied constant does not depend on $u$. Note the missing $(1-y)(1-z)$ factors in the denominator. 

Then we continue like in the old proof, yet combine it with the ideas of~\cite{D1994} to capture the perturbation by $u$ in a Gaussian integral. Finally, we get
\begin{align*}
	\widetilde{\mathfrak m}^{(2)}_{\lambda,\lambda+u_1,\lambda+u_2} = \frac{7}{3} \frac{4^\lambda}{\pi \lambda} e^{-\frac{u_1^2+u_2^2}{\lambda}} \left( 1 + \LandauO\left(\frac{u}{\lambda}\right)\right).
\end{align*}

It remains to compute the partial sums in the second and third variable. However, this is analogous to the proof of Lemma~\ref{lem:asyfirstmom}, and we get our final result.
\end{proof}

\subsection{Small second moment}
Finally, we show that~\eqref{eq:Xlkdefspecific} is the proper rescaling and that it possesses a small second moment. We define the random variable $\hat{X}_{\lambda,k}$ uniformly distributed on $\bigl\{2^\lambda,\ldots,2^{\lambda+1}-1\bigr\}$ according to the following ansatz
\begin{align*}
	\hat{X}_{\lambda,k} \coloneqq n \mapsto \widetilde\Theta(k,n) - v_{\lambda,k} n - w_{\lambda,k},
\end{align*}
where $v_{\lambda,k}$ and $w_{\lambda,k}$ are independent of $n$. Note that this random variable encodes a more general rescaling as the constant term $w_{\lambda,k}$ is also present. We will see that its size (if properly chosen) does not influence the result. 

The second moment of $\hat{X}_{\lambda,k}$ is equal to
\begin{align*}
	\E(\hat{X}_{\lambda,k}^2) &= \E\left(\widetilde\Theta(k,n)^2\right) -2 v_{\lambda,k} \E\left(n \widetilde\Theta(k,n)\right)  - 2 w_{\lambda,k}\E\left(\widetilde\Theta(k,n)\right)\\
	                    &+ v_{\lambda,k}^2 \E\left(n^2\right) + 2 v_{\lambda,k} w_{\lambda,k} \E\left(n\right) + w_{\lambda,k}^2 .
\end{align*}
As $\hat{X}_{\lambda,k}$ is uniformly distributed on $\bigl\{2^\lambda,\ldots,2^{\lambda+1}-1\bigr\}$ we directly compute
\begin{align*}
	\E\left(n\right) &= \frac{3}{2} 2^\lambda - \frac{1}{2}, \quad \text{ and }\\
	\E\left(n^2 \right) &= \frac{7}{3} 4^\lambda - \frac{3}{2} 2^\lambda + \frac{1}{6}.
\end{align*}
Therefore, Lemmata~\ref{lem:asyfirstmom}, \ref{lem:asymixedmom}, and \ref{lem:asysecondmom} show that for 
\begin{align*}
	v_{\lambda,k} &= \LandauO\left(1\right) &\text{ and } && 
	w_{\lambda,k} &= \LandauO\left(\frac{2^\lambda}{\sqrt{\lambda}}\right)
\end{align*}
the dominant terms of the second moment $\E(\hat{X}_{\lambda,k}^2)$ have an exponential growth of order $4^\lambda$. In particular, we get the explicit result
\begin{align*}
	\E(\hat{X}_{\lambda,k}^2) &= \frac{7}{3} 4^{\lambda} \left( v_{\lambda,k} - \Phi \left(\frac{u}{\sqrt{\lambda}}\right) \right)^2 + \LandauO\left(\frac{4^\lambda}{\sqrt{\lambda}}\right).
\end{align*}
Thus, we see that if we choose 
\begin{align*}
	v_{\lambda,k} &= \Phi\left(\frac{u}{\sqrt{\lambda}}\right)\left(1 + \LandauO\left(\frac{1}{\lambda^{1/4}}\right)\right) & \text{ and } &&
	w_{\lambda,k} &= \LandauO\left(\frac{2^\lambda}{\sqrt{\lambda}}\right).
\end{align*}
it gives 
\begin{align*}
	\E\left(X_{\lambda,k}^2\right) = \LandauO\left(\frac{4^\lambda}{\sqrt{\lambda}}\right).
\end{align*}
This proves Proposition~\ref{prp_closeness}. As mentioned before, this proof also gives the possible size for a constant (with respect to $n$) term $w_{\lambda,k}$ such that the result still holds. 

Furthermore, it is also possible to specify the constant $C$ in the speed of convergence, and get further error terms automatically. Yet, the computations quickly become cumbersome.


\section{Acknowledgements}
We wish to thank Michael Drmota for pointing out his paper~\cite{D1994}
and for several fruitful discussions.

\bibliographystyle{plain}
\bibliography{SW}
\end{document}